\documentclass[11pt]{amsart}

\usepackage{amsfonts}
\usepackage{amssymb}
\usepackage{amsmath}
\usepackage{amsthm}
\usepackage{verbatim}

\usepackage{hyperref}

    \normalbaselines                          %

\newtheorem{thm}{Theorem}
\newtheorem{lem}{Lemma}
\newtheorem{prop}{Proposition}

\newtheorem*{conj*}{Conjecture}

\theoremstyle{remark}
\newtheorem{rem}
{Remark}
\newtheorem*{rem*}{Remark}
\theoremstyle{definition} 
\newtheorem{defn}{Definition}
\newtheorem*{defn*}{Definition}

\newcommand{\bb}{\bigskip}

\newcommand{\n}{\noindent}

\newcommand{\e}{\varepsilon}
\newcommand{\p}{\varphi}

\newcommand{\Om}{\Omega}

\newcommand{\R}{\mathbb{R}}
\newcommand{\Rd}{\mathbb{R}^{n_1} \times...\times \mathbb{R}^{n_d}}
\newcommand{\Z}{\mathbb{Z}}

\newcommand{\D}{\mathcal{D}}

\newcommand{\cD}{\mathcal{D}}
\newcommand{\cF}{\mathcal{F}}

\newcommand{\bmo}{{\rm BMO}}
\newcommand{\vmo}{{\rm VMO}}

\newcommand{\cR}{\mathcal{R}}

\newcommand{\ci}[1]{_{ {}_{\scriptstyle #1}}}

\newcommand{\ti}[1]{_{\scriptstyle \text{\rm #1}}}

\def\XXint#1#2#3{{\setbox0=\hbox{$#1{#2#3}{\int}$}
     \vcenter{\hbox{$#2#3$}}\kern-.5\wd0}}

\begin{document}

\title{Weak-star convergence in multiparameter Hardy spaces}
\author{Jill Pipher}
\thanks{The first author  is  supported by the NSF grant DMS-0901139
}
\address{Jill Pipher, Department of Mathematics, Brown University, 151 Thayer
Str./Box 1917,      
 Providence, RI  02912, USA }

 \author{Sergei Treil}
 \thanks{The second author  is  supported by the NSF grant DMS-0800876.
}

\address{Sergei Treil, Department of Mathematics, Brown University, 151 Thayer
Str./Box 1917,      
 Providence, RI  02912, USA }
\email{treil@math.brown.edu}
\urladdr{http://www.math.brown.edu/\~{}treil}

\begin{abstract}
\n We prove a multiparameter version of
a classical theorem of Jones and Journ\'e on weak-star convergence
in the Hardy space $H^1$.
\end{abstract}

\maketitle

\section{Introduction}

It is a well known and classical result that the
commutator of a singular integral with the operator
of multiplication by a function in $\bmo$, the space of bounded
mean oscillation, is bounded in $L^p$, for $1<p< \infty$.  
The first proof appeared in \cite{CRW} (see also \cite{Ja}), and there are
now generalizations of this result to the bidisc and polydisc
(\cite{FL}, \cite{LT}, \cite{LPPW}).
Since $BMO$ is the dual of the Hardy space $H^1$ of functions
whose Poisson maximal function (or square function) belongs to $L^1$,
one can formulate a dual version
of the commutator result. This dual formulation asserts that certain quantities involving products
and sums of Riesz transforms (or more general singular integrals) belong to $H^1$.
For example, if $R_j$ denotes the $j$th Riesz transform, the quantity
$gR_jf + fR_jg$ belongs to $H^1(\R^d)$, whenever $f,g \in L^2(\R^d)$.
Each of the summands in this quantity clearly belongs to $L^1$,
but it is the special form of this sum which puts it into $H^1$.
The space $\vmo$ is the predual of $H^1$, and this gives $H^1$ a richer
structure than $L^1$.

In \cite{CLMS}, a much more general approach was developed.
There, the authors showed that a variety of expressions with
a special form of cancellation (the div-curl quantities) belong to some Hardy space $H^p$.
Their approach paved the way for a striking collection of extensions of the
theory of compensated compactness in partial differential equations.
A result of P. Jones and J.-L. Journ\'e concerning weak convergence in
the Hardy space $H^1$
(\cite{JJ}) was essential.

\medskip

In this paper we prove the multiparamater analog of this theorem. That is,
if $\Rd$ denotes a $d$-parameter product space, where $n_i \geq 1$, we have
the following:

\begin{thm}[Jones - Journ\'e in the multiparameter setting]
\label{thm:weak}
  Suppose that $\{f_n\}$ is a sequence of $H^1(\Rd)$ functions
  such that $\Vert f_n \Vert_{H^1} \leq 1$, for all $n$, and such that
  $f_n(x) \rightarrow f(x)$ for almost every $x \in \Rd$.
  Then $f \in H^1(\Rd)$, $\Vert f \Vert_{H^1} \leq 1$, and $f_n \xrightarrow{w*} f$, i.e. for any
  $\p \in \vmo(\Rd)$,
  $$ \int_{\Rd} f_n \p dx \rightarrow  \int_{\Rd} f\p dx.$$

\end{thm}

The inherent difficulty in working with the multiparameter $\bmo$ and
$\vmo$ spaces is that
the definitions require one to deal with
arbitrary open sets, as opposed to
intervals or products of intervals.

The paper is organized as follows. We recall (in section~\ref{sec:definitions}) some definitions
and the results from prior work
which are required in the proof.
The proof in Section 3 follows the template provided in \cite{JJ}, but
some new ideas, see Lemmas \ref{Lemma:LemmaA} and \ref{Lemma:LemmaB} below, are needed to
get from the one-parameter to multiparameter case.

\section{Definitions}
\label{sec:definitions}

\begin{defn}
A real-valued function $f\in L^1_{\text{loc}}(\Rd)$ is in the space $bmo$ (called
``little bmo" in the literature), if its $bmo$ norm
is finite:
\begin{equation}\label{equation:defbmo}
  \Vert f \Vert_{bmo} := \sup_R \, \frac{1}{|R|} \,
  \int_R \, |f(x) - (f)_R| \, dx < \infty.
\end{equation}
Here $(f)_R = \frac{1}{|R|}\int_R f(x) \, dx$ is the average value of $f$
on the rectangle~$R=Q_1 \times ... \times Q_d \subset \Rd$, where
$Q_i \subset \R^{n_i}$.
\end{defn}

When $d=1$, this is the classical $\bmo$ space, the dual of the
Hardy space $H^1$.
In the multiparamater setting, $bmo$ is one of several possible generalizations of
the one-parameter $\bmo$ space. It is not hard to see that a function
$f$ belongs to $bmo(\Rd)$ if and only if, for all indices $i$, $f$ belongs  to the classical
one-parameter spaces
$\bmo(R^{n_i})$ (with the uniform estimates on the norms), with the other variables fixed.
Moreover, by the John-Nirenberg theorem for the classical one-parameter $\bmo$ space,
the $L^1$ norm in \eqref{equation:defbmo} can be replaced by the $L^p$ norm, $1\le p<\infty$.
In this paper we will use the equivalent norm with $p=2$,
\begin{equation}
\label{equation:defbmo2}    
\Vert f \Vert_{bmo} := \sup_R \,\left(  \frac{1}{|R|} \,
  \int_R \, |f(x) - (f)_R|^2 \, dx \right)^{1/2} < \infty.
\end{equation}

However, the true analog of
$\bmo$ - in the sense of duality with the multiparameter
Hardy space, and boundedness of singular integrals - is the
product $\bmo$ space
which was defined and characterized by
S.-Y. Chang and R. Fefferman (\cite{C} and \cite{CF}).
The space $bmo$ defined above is strictly smaller than this product $\bmo$ space.
(See \cite{FS}.)

\bigskip

The dyadic lattice $\D(\R^n)$
in $\R^n$ is constructed as follows: for each $k\in \Z$ consider the cube $[0,2^k)^n$
and all of its shifts by elements of $\R^n$ whose coordinates are $j2^k$, $j\in \Z$; then take the union over all $k\in \Z$.

\begin{defn}[Expectation and Difference operators]
Let $E_k$ denote the averaging operator over cubes $Q \in \D(R^n)$ of side length $2^k$:
$E_kf(x) = 1/|Q|\int_Q f(y) dy$, if $Q$ has side length $2^k$ and contains $x$.
If $Q$ has side length $2^k$, then $E_Qf(x) = E_kf(x)\chi_E(x)$.
Set $\Delta_k = E_{k-1} - E_k$, and $\Delta_Qf(x) = \Delta_kf(x)\chi_Q(x)$, when
$Q$ has side length $2^k$.
\end{defn}

\begin{defn}[Square functions, dyadic $H^1$]
 For a ``dyadic rectangle'' $R=Q_1 \times ... \times Q_d$, $ Q_i\in \cD(\R^{n_i})$
define the multiparameter difference operator
$\Delta_R = \Delta_{Q_1} \otimes ... \otimes \Delta_{Q_d}$. We use the symbol $\otimes$ to emphasize that the difference operators $\Delta_{Q_i}$ act on independent variables $x_i\in \R^{n_i}$.

\medskip

Here we use the same notation for the one-parameter difference operator and for the multiparameter one;
cubes are always subsets of the ``building blocks'' $\R^{n_i}$, and the ``rectangles'' are the Cartesian products of cubes.
Even if the size of all cubes $Q_i$ is the same, we will call the product $Q_1\times Q_2\times\ldots\times Q_d$ a ``rectangle''.   

Denote by $\cR=\cR ( \Rd )$ the collection of all ``dyadic rectangles''.

The (multiparameter) square function of $f$ in $\Rd$ is defined as
\begin{equation}
Sf(x) = \left(\sum_{R\in\cR} |\Delta_Rf(x)|^2\right)^{1/2}
\end{equation}

A function $f$ belongs to the Hardy space $H^1$ if its norm
$\Vert f \Vert_{H^1} := \Vert Sf \Vert_{L^1}$ is finite.

\end{defn}

\begin{rem}
\label{r1}
Similarly to the one-parameter case for $f\in L^2(\Rd)$
\[
\|f\|^2_2 = \sum_{R\in \cR} \|\Delta_R f \|^2_2 .
\]
This fact is well-known in one-parameter situation:
the general case can be easily obtained by iterating the one-parameter case.
\end{rem}

\begin{defn}[Dyadic product $\bmo$]
  A function $f$ belongs to $ \bmo_d$
  if there exists a constant~$C$ such that for every open set $\Omega \subset \Rd$,
  \begin{equation}\label{eqn:bmodyadic}
    \sum_{R \subset \Omega} \,\Vert \Delta_Rf \Vert_2^2  \leq C|\Omega|.
  \end{equation}
  (See \cite{B}.)
\end{defn}

The dyadic $H^1$ and $\bmo$ spaces can be defined
in terms of a Carleson packing condition using
the product system of Haar wavelets. (See \cite{BP} for
example.)
The same Carleson packing condition, but using
a  basis of smooth wavelets, such as the Meyer wavelets,
defines the product $\bmo$ space, which is the dual of product $H^1$.
We refer to \cite{CF} for the precise definition, and the
duality theorem. Here, we shall only need the following relationship
between product $\bmo$ and its dyadic counterpart, $\bmo_d$:

\begin{prop}\label{prop:averaging}
If $\p$ and all its translates belong to the product $\bmo_d$, with
uniform bounds on their $\bmo_d$ norms, then $\p$ belongs
to the product space $\bmo$.
\end{prop}

This statement is trivial in one-parameter settings. In multiparameter situation, it can be treated as a special case of the so-called ``BMO from dyadic BMO'' result (which is a significantly stronger statement), see \cite{PW}, \cite[Remark 0.5]{Tr}.

Namely,   let us consider all translations $\cD_\omega$ of the standard dyadic lattice $\cD$. If we have a measurable family of functions $\p_\omega$, such that each  $\p_\omega$ belongs to $\bmo_d$  with respect to the corresponding lattice $\cD_\omega$ (with the uniform estimate of the norm), then the average (over all $\omega$) of $\p_\omega$ is a BMO function.

Here we do not explain how the average over all $\omega$ is computed, since in our situation $\p_\omega=\p$, so the average is also $\p$; see \cite{PW}, \cite[Remark 0.5]{Tr} for more details.

Note, that the ``BMO from dyadic BMO'' statement is non-trivial even in one-parameter setting, see \cite{D, GJ} for the proof in this case.

\begin{defn}
The product $\vmo$ space is
the closure of the $C^{\infty}$ functions in the product $\bmo$ norm.
\end{defn}

{\bf Remark}. As in the classical one-parameter setting,
the product $\vmo$ space is the predual of $H^1$. (See \cite{LTW}.)

\medskip

\section{Proof of the main result.}

\begin{defn}
If $\Om \subset \Rd$ is an open set, and $x_1 \in \R^{n_1}$, the ``slice"
$\Omega_{x_1}$ is the $(n_2+...+n_d) -$dimensional set:
$$\{x' \in \R^{n_2} \times ... \times \R^{n_d}: (x_1,x') \in \Om \}.$$
The slices $\Omega_{x_i}$, for $i=2,...,n$ are defined similarly.
\end{defn}

Let $\cF \subset\cR$ be a family of ``dyadic rectangles''
 $R=Q_1 \times... \times Q_d$ , $Q_i \in \cD( \R^{n_i})$.
For $x_1\in \R^{n_1}$ let $\cF_{x_1}$ denote the $x_1$ ``slice'' of the family $\cF$, i.e.~the set of all ``rectangles'' $R'= Q_2\times Q_3\times \ldots \times Q_d\subset \R^{n_2}\times\R^{n_3}\times\ldots\times \R^{n_d}$ for which there exists a cube $Q_1\subset \R^{n_1}$ such that $x_1\in Q_1$ and
\[
R = Q_1\times R' = Q_1\times Q_2 \times \ldots \times Q_{d} \in\cF.
\]

\begin{lem}\label{Lemma:LemmaA}
Let $\cF\subset \cR$.
Then
\begin{equation}
\label{eq:A}
\sum_{R \in \cF} \Vert\Delta_Rf\Vert^2 \le \int_{\R^{n_1}} \sum_{R' \in \cF_{x_1}} \Vert\Delta_{R'}f(x_1,\,.\,)\Vert^2 dx_1.
\end{equation}
\end{lem}

\medskip

\begin{proof}
For $x_1\in \R^{n_1}$ and $x' \in \R^{n_2}\times\R^{n_3}\times\ldots\times \R^{n_d}$  set
\[
\tilde{f}(x_1,\,.\,) = \sum_{R' \in \cR_{x_1}}  \Delta_{R'}f(x_1,\, . \,) .
\]
Then, for 
fixed $x_1\in \R^{n_1}$,
we observe that when $Q_1 \times R' \in \cF$ and $x_1 \in Q_1$, then
\begin{equation}
\label{eq:relate}
\Delta_{R'}f(x_1,\,.\,) = \Delta_{R'}\widetilde{f}(x_1,\,.\,)..
\end{equation}

\n This is because our assumptions $x_1 \in Q_1$ and  
$Q_1 \times R' \in \cF$ mean exactly that $R' \in \cF_{x_1}$.
Thus, using the fact that $\Delta_{Q_1\times R'} =\Delta_{Q_1} \otimes \Delta_{R'}$ we get that
\[
\Delta_{Q_1\times R'} f = \Delta_{Q_1\times R'}\widetilde f
\]
and so
\begin{align*}
\sum_{Q_1 \times R' \in \cF} \Vert\Delta_{Q_1\times R'}f\Vert^2_2 &=
\sum_{Q_1 \times R' \in \cF} \Vert \Delta_{Q_1\times R'}\tilde{f}\Vert^2_2\\
\leq \Vert \tilde{f} \Vert^2_2
&= \int_{\R^{n_1}} \Vert \sum_{R' \in \cF_{x_1}} \Delta_{R'}\widetilde f(x_1,.)\Vert^2_2 dx_1\\
&= \int_{\R^{n_1}} \Vert \sum_{R' \in \cF_{x_1}} \Delta_{R'}f(x_1,.)\Vert^2_2 dx_1\\
&= \int_{\R^{n_1}} \sum_{R' \in \cF_{x_1}} \Vert \Delta_{R'}f \Vert^2_2 dx_1.
\end{align*}

\end{proof}

\bigskip

\begin{lem}\label{Lemma:LemmaB}
Suppose
$\p \in C^{1}(\Rd)$, $\|\p\|_\infty\le1$, $\|\nabla_{x_i} \p(x)\|_{\ell^1}\le1$, $i=1, 2, \ldots, d$, and $b$ is a bounded
function with $\Vert b \Vert_{\infty} \leq 1.$
Then, for any $\alpha < 1$, and any open $\Omega\subset \Rd$,

\begin{equation}\label{equation:induction}
\sum_{R\in \cR: R \subset \Om, |R| \leq \alpha} \Vert\Delta_R(\p b)\Vert^2_2 \leq
2d!(\Vert b \Vert^2_{bmo} + \alpha^{2/n})|\Om|,
\end{equation}
where $n=n_1+ .. + n_d$, and $\Vert f \Vert_{bmo}$ is defined by \eqref{equation:defbmo2}.
\end{lem}

\begin{proof}
The proof is by induction on $d$:
the base case for one-parameter BMO ($d=1$) was proven in \cite{JJ}, but we'll give the
short argument here for the sake of completeness.

When $d=1$, it suffices to prove (\ref{equation:induction}) for $\Omega = Q_0$, where
$Q_0 \subset \R^{n_1}$ is a dyadic cube. The ``rectangles"' $R$ are themselves dyadic cubes (which
we now denote by $Q$), and
by subdividing $Q_0$ into smaller dyadic cubes if necessary, we may without loss of generality
assume that $|Q_0| \leq \alpha$. Then  we see that
\begin{eqnarray*}
    \sum_{Q \subset Q_0} \Vert\Delta_Q(\p b)\Vert^2_2
    & = &
    \int_{Q_0} \, |\p(x) b(x) - (\p b)_{Q_0}|^2 \, dx \\
    & \le &
    \int_{Q_0} \, |\p(x) b(x) - \p_{Q_0} b_{Q_0}|^2 \, dx \\
    & \le & 2\left( \int_{Q_0} \, |\p(x) b(x) - \p(x) b_{Q_0}|^2 \, dx  
    + \int_{Q_0} \, |\p(x) b_{Q_0} - \p_{Q_0} b_{Q_0}|^2 \, dx \right)
\end{eqnarray*}

\n On the one hand, by the pointwise bound on $\p$,
\[
\int_{Q_0} \, |\p (x) b(x) - \p(x) b_{Q_0}|^2 \, dx \leq \Vert b \Vert_{bmo}^2 |Q_0|,
\]
and using the pointwise bounds on $b$ (so $|b_{Q-0}|\le 1$) and on derivatives of $\p$ (so $|\p(x) - \p_{Q_0}|\le (\alpha)^{1/n_1}$ for $x\in Q_0$), we get
\[
\int_{Q_0} \, |b_{Q_0}(\p(x) - \p_{Q_0})|^2 \, dx \leq (\alpha)^{2/n_1}|Q_0|.
\]
Combining the above 3  estimates we get \eqref{equation:induction} with $d=1$.

For the induction step,
we'll use the notation $\R^{n_1} \times ... \widehat{\R}^{n_i} \times ... \R^{n_d}$
to denote the $d-1$ fold product of
the $\Rd$ with $\R^{n_i}$ missing, and similar notation for a $d-1$ fold
product of cubes with one cube $Q_i$ missing.

Suppose now that $R = Q_1 \times ... \times Q_d$ is a rectangle
in $\Rd$ with $|R| < \alpha$. Then there exists an $i$ such that
the $d-1$ dimensional rectangle $R_i'=Q_1 \times ...\times...\hat{Q_i}... \times Q_d$
has volume $|R'_i| < \alpha^{N_i}$, where
$N_i = (n_1+...\hat{n_i}+...+n_d)/(n_1+...+n_d)$.
Indeed,  if not,
\[
|R|^{d-1} = \prod_{i=1}^d |R'_i|
 >  \prod_{i=1}^d  \alpha^{N_i} = \alpha^{d-1}
\]
contradicting the assumption $|R|<\alpha$.

Thus each ``rectangle'' $R \subset \Om$, $|R|<\alpha$ satisfies this condition for at least one index $i =1,...,d$. Therefore, the collection $\cF=\{R\in \cR: R\subset \Omega, |R|<\alpha\}$ can be represented as a union $\cF = \cup_{i=1}^d \cF^i$, where $\cF^i:= \{ R\in \cF: R_i'<  \alpha^{N_i} \}$. (Note, that $\cF^i$s are not necessarily disjoint.)

Applying Lemma \ref{Lemma:LemmaA} (with $x_1$ replaced by $x_i$) to each collection $\cF^i$, we see that
\begin{equation*}
\sum_{R \in \cF} \Vert\Delta_R(\p b)\Vert^2_2 \le
\sum_{i=1}^{d}\int_{\R^{n_i}} \sum_{R' \in \cF^i_{x_i}} \Vert\Delta_{R'}(\p b)(x_i,\,.\,)\Vert^2_2 dx_i .
\end{equation*}
Note that $\cF^i_{x_i} \subset \{R'\in \cR(\R^{n_1} \times ... \widehat{\R}^{n_i} \times ... \R^{n_d}): R'\subset  \Omega_{x_i}, |R'|\le \alpha^{N_i} \}$, so
by the induction step with $\tilde{n_i} = n_1+...n_{i-1}+n_{i+1}+...+n_d)$ instead of $n$ and $d-1$ instead of $d$, we get
\begin{eqnarray*}
\int_{\R^{n_i}} \sum_{R' \in \cF^i_{x_i}, |R'| < \alpha^{N_i}} \Vert \Delta_{R'}(\p b)(x_i,.)\Vert^2_2
&\le & 2 (d-1)!
\int_{x_i} (\Vert b \Vert_{bmo}^2 + (\alpha^{N_i})^{2/\tilde{n_i}}) |E_{x_i}| dx_i\\
&=& 2(d-1)!(\Vert b \Vert_{bmo}^2 + \alpha^{2/n}) |\Om|.
\end{eqnarray*}

Here we have also used the (trivial) fact that the $bmo(\R^{n_1} \times ... \widehat{\R}^{n_i} \times ... \R^{n_d})$
norm of $b$ is bounded
by $\Vert b \Vert_{bmo(\Rd)}$.

Adding estimates for all $i=1, 2, \ldots, d$ we get the conclusion of the lemma.
\end{proof}

\bigskip

We will require the following fact about $bmo$ functions.

\begin{lem}\label{lem:max}
If $f$ and $g$ belong to $bmo$, then $\text{max}\{f,g\}$ also belongs to $bmo$.
\end{lem}

\begin{proof}
The proof is exactly as in the one-parameter setting, since the space
$bmo$ is defined by averages over rectangles. That is,
for any rectangle $R$, we have
\begin{equation*}
\frac{1}{|R|} \int_R |\, |f(x)| - |f|_R \,|dx \leq \frac{1}{|R|} \int_R \,|f(x) - f_R| \, dx
\end{equation*}
and $\max\{f,g\} = (|f-g| + f + g)/2$.  
\end{proof}

\bigskip

\begin{lem}\label{lem:littlebmo}
Let $E \subset \Rd$ be a set of finite measure, and
let $\delta > 0$ be a given parameter. Then there exists
a function $\tau \in bmo$ such that
$\tau = 1$ on $E$,
$\Vert \tau \Vert_{bmo} < C_1 \delta$, and
$|\operatorname{supp}\tau | < C_2 e^{2/\delta} |E|$, where $C_1$ and $C_2$ are some absolute constants. 
\end{lem}

\begin{proof}
Recall that a weight $w$ belongs to the $A_1$ class if
there exists a constant $C$ such that for all $x$, $Mw(x) \leq Cw(x)$.
Here, if $M$ is the Hardy-Littlewood maximal function then this
is the usual class of (one-parameter) Muckenhoupt weights. And if
$M$ is the strong maximal function where the averages are taken
over arbitrary rectangles in $\Rd$, then this is the multiparamater
$A_1$ class. (See \cite{FP} for some basic facts about product $A_p$ weights.)

We define the following $A_1$ weight, with $M^{(k)}$ denoting the
k-fold iteration of the strong
maximal function:
\begin{equation*}
m(x) = K^{-1}\sum_{k=0}^{\infty} c^k M^{(k)}\chi\ci E(x).
\end{equation*}
\n where $K=\sum_k c^k$, and $c>0$ is chosen to insure the convergence of the series. Namely, we chose $c$ such that $\|cM\|<1$, i.e.~that for some $q<1$, we have $c \|M f \|_2 \le q\|f\|_2$ for all $f\in L^2$.  

Then  $\Vert m \Vert\ci{L^2} \leq C \|\chi\ci E\|\ci{L^2} = C |E|^{1/2}$.
Observe that $m=1$ a.e. on $E$, and $m \leq 1$ a.e. outside of $E$.

\bigskip

Define, as in \cite{JJ}, following \cite{CR}, the function
\begin{equation*}
\tau(x) = \max\{0, 1 + \delta \log m(x) \}.
\end{equation*}

The function $\tau$ belongs to $bmo$ and also satisfies $\tau = 1$ a.e. on $E$.
However, $\tau$ has small $bmo$ norm: $\Vert \tau \Vert_{bmo} \lesssim \delta$.
This follows from Lemma \ref{lem:max} and the fact that for any $A_1$ weight $w$, $\log w$ belongs to $bmo$,
which is proved exactly as in the one-parameter setting. (See, for example, \cite{G}.)
 
The estimate for the size of the support of $\tau$ follows from
Tchebychev's theorem and the  estimate $\Vert m \Vert\ci{L^2} \leq  C |E|^{1/2}$.   
\end{proof}

\bb

We now prove Theorem \ref{thm:weak}.
\begin{proof}[Proof of Theorem \ref{thm:weak}]
First notice, that since $C^{\infty}_0$ is dense in VMO, it is sufficient to prove Theorem \ref{thm:weak} for $\p\in C^{\infty}_0$. 

Because $f_n \rightarrow f$ a.e., and $\Vert f_n \Vert_{H^1} \leq 1$, we have
$\Vert f \Vert|_{L^1} \leq 1$ by Fatou's Lemma. Choose a $\p \in C^{\infty}_0$,
normalized to have $|\p| \leq 1$, $\Vert \p \Vert_{L^1} \leq 1$.
Let $\e > 0$ be fixed.

We need to show that, for $n$ sufficiently large,
\begin{equation}\label{eqn:estimate}
\Bigl|\int f_n \p dx  -  \int f \p dx \Bigr| < C\e,
\end{equation}
where $C$ is some absolute constant. 

We now fix an $\eta$ to be determined, and define
\begin{equation*}
E_n = \{x \in \text{supp}\,\p: |f_n(x) - f(x)| > \eta \}
\end{equation*}

Choose $n$ sufficiently large that $|E_n| < \eta$.
Define $\tau$ as in Lemma \ref{lem:littlebmo}, relative to the set $E_n$. 
Then, if $\eta$ is chosen sufficiently small, since $\|f\|_{L^1}\le 1$, we will have
\begin{equation*}
\left|\int_{\operatorname{supp}\tau} f dx < \e \right|  .
\end{equation*}

Then we break up the integral in (\ref{eqn:estimate}) as
\begin{equation}\label{eqn:sum}
\left| \int (f-f_n) dx \right| \leq \left| \int (f-f_n)\p(1-\tau) dx \right|
+ \left|\int(f-f_n)\p\tau dx\right| .
\end{equation}

In the complement of $E_n$, $|f-f_n| < \eta$, so the first integral
on the left hand side of (\ref{eqn:sum})
is bounded by $\eta \Vert \p \Vert_{L^1}$, which in turn is less than
$\e$ if $\tau$ is appropriately small.

The second integral in (\ref{eqn:sum}) is bounded by
\begin{equation*}
\int_T |f\p| dx + \left| \int f_n \p \tau dx\right|, 
\end{equation*}
 and since
\begin{equation*}
\int_T |f\p| dx < \e , 
\end{equation*}
the proof is completed by showing that $\Vert \p \tau \Vert_{\bmo} \lesssim \e$.

We will show that the dyadic $\bmo$ norm of $\p \tau$ has the required estimate, and observe
that the same proof shows that any translate of $\p \tau$ is in dyadic $\bmo$ with the
same bound. The estimate on the product $\bmo$ norm will follow from
Proposition \ref{prop:averaging}.

Fix an arbitrary open set $\Om \subset \Rd$ and consider two cases:

\medskip

\n (i) $|\Om| \leq \alpha$, where $\alpha>0$ will be chosen in a moment. 

\n In this case, all rectangles contained in $\Om$ have size less than $\alpha$,
and Lemma \ref{Lemma:LemmaB} gives
\begin{equation*}
\sum_{R \subset \Om, |R| \leq \alpha} \Vert\Delta_R(\p \tau)\Vert^2_2 \leq
C(\Vert \p \Vert^2_{bmo} + \alpha^{2/n})|\Om|,
\end{equation*}
for $n=n_1+...+n_d$. With the appropriate choice of  $\alpha$ and $\delta$ from Lemma \ref{lem:littlebmo} (so $\Vert \p \Vert^2_{bmo} \le \delta^2$),
this will be smaller than $\e |\Om|$. 

Note that $\eta$ does not appear in the above estimate, so it holds for all $\eta$ ($\eta$ appears in the estimate of $|\operatorname{supp}\tau|$, but we do not use this quantity in the estimate).

\pagebreak

\medskip

\n (ii) $|\Om| > \alpha$ ($\alpha$ and $\delta$ are already chosen). 

\n In this case, using the estimates $\| \p\tau\|_\infty\le 1$ and $|\operatorname{supp}\tau| \le C_2 e^{2/\delta} \eta $, we get
\begin{eqnarray*}
\sum_{R \subset \Om} \Vert\Delta_R(\p \tau)\Vert^2_2 &\leq& \int |\p \tau|^2 dx\\
&\leq& \Vert \p \tau \Vert_{\infty}^2 |\operatorname{supp}\tau|\\
&\leq& C_2 \eta e^{2/\delta} \\
&\leq& C_2 \eta e^{2/\delta} \frac{|\Om|}{\alpha}, 
\end{eqnarray*}
and the last quantity is bounded by $\e |\Om|$ if  $\eta$ is small enough.
\end{proof}

\begin{rem*}
 Note that the multiparameter version of the Jones-Journ\'e theorem cannot be obtained by \emph{trivial} iteration of the original one-parameter version. Namely, the function $\p\tau$ in the proof does not have small {\it bmo} norm, so we need to use the norm of the product $\bmo$. Lemmas \ref{Lemma:LemmaA} and \ref{Lemma:LemmaB} are necessary to perform this iteration.  
\end{rem*}

\begin{rem*}
As it was mentioned in \cite{JJ}, it is easy to see that the analogue of the main result does not hold for $L^1$ functions: it is easy to construct a sequence of $L^1(\R^N)$ functions $f_n$ converging (in the weak* topology of the space of measures $M(\R^N)$) to a singular measure and such that $f_n\to0$ a.e.

Moreover, picking a sequence of discrete measures $\mu_n$, converging (in the weak* topology of $M(\R^N)$) to a given $f\in L^1(\R^N)$, and then approximating the measures $\mu_n$ by absolutely continuous measures with densities $f_n$ (recall that the weak* topology of $M(\R^N)$ is metriziable on any bounded set),  we get that $f_n \xrightarrow{w*} f$ in $M(\R^N)$. One can definitely pick a sequence $f_n$ such that  $f_n\to0$ a.e., which gives us even more striking counterexample.

On the other hand, the analogue of Theorem \ref{thm:weak} holds for any reflexive function space $X$ of locally integrable functions (so convergence in $X$ implies the convergence in $L^1\ti{loc}$). Namely, if $\sup_n\Vert{f_n}\Vert<\infty$ and  $f_n\to f$ a.e., then $f_n\to f$ in the weak (which is the same as weak*) topology of $X$. This is a simple exercise in basic functional analysis, we leave the details to the reader.

The space $H^1$ however is not reflexive: it is only a dual (of VMO). So, maybe an analogue of Theorem \ref{thm:weak} is true for any space of function which is dual to some space. It would be interesting to prove or disprove the following conjecture.
\end{rem*}

\begin{conj*}
Let $X$ be a Banach function space $X$ of locally integrable functions (so convergence in $X$ implies the convergence in $L^1\ti{loc}$), which is dual to some Banach function space $X_*$ (with respect to the natural duality). If $f_n\in X$ such that $\sup_n \|f_n\|<\infty$ and $f_n\to f$ a.e., then $f\in X$ and $f_n\to f$ in weak* topology of $X$.
\end{conj*}

If it helps to prove the conjecture, one can assume that $X$ is a ``reasonable'' space: for example that $C_0^\infty$ is dense in $X$ and/or $X^*$, etc. The result under these (or similar) additional assumptions will still be extremely interesting.

\end{document}